%% file: main.tex
\documentclass[11pt]{article}
\usepackage[letterpaper, margin=1.0in]{geometry}

\usepackage{times}
\usepackage{microtype}
\usepackage{fullpage}
\usepackage[numbers,sort]{natbib}

\usepackage[dvipsnames]{xcolor} 
\usepackage{tikz}
\usetikzlibrary{shapes}
\usetikzlibrary{positioning}
\usetikzlibrary{plotmarks}

\usepackage{microtype}
\usepackage{graphicx}
\usepackage{enumitem,caption}
\usepackage{booktabs} 
\usepackage{multirow}
\usepackage{adjustbox}

\usepackage{stmaryrd}
\usepackage{graphicx}
\usepackage{caption}
\usepackage{subcaption}

\usepackage{textcomp} 
\usepackage{quoting}
\usepackage{titletoc} 
\usepackage{fontawesome} 

\usepackage[colorlinks = true,
            linkcolor = blue,
            urlcolor  = blue,
            citecolor = blue,
            anchorcolor = blue]{hyperref}

\usepackage{algorithm}
\usepackage[noend]{algpseudocode}
\algrenewcommand\algorithmicrequire{\textbf{Input:}}
\algrenewcommand\algorithmicensure{\textbf{Output:}}

\renewcommand{\epsilon}{\varepsilon}

\newlength{\continueindent}
\usepackage{etoolbox}
\makeatletter
\newcommand*{\ALG@customparshape}{\parshape 2 \leftmargin \linewidth \dimexpr\ALG@tlm+\continueindent\relax \dimexpr\linewidth+\leftmargin-\ALG@tlm-\continueindent\relax}
\apptocmd{\ALG@beginblock}{\ALG@customparshape}{}{\errmessage{failed to patch}}
\makeatother

\algblock{ParFor}{EndParFor}
\algnewcommand\algorithmicpardo{\textbf{in parallel do}}
\algrenewtext{ParFor}[1]{\algorithmicfor\ #1\ \algorithmicpardo}
\algrenewtext{EndParFor}{\textbf{end for}}
\algtext*{EndParFor}
\algdef{SE}[SUBALG]{Indent}{EndIndent}{}{\algorithmicend\ }%
\algtext*{Indent}
\algtext*{EndIndent}

\providecommand{\keywords}[1]{\textbf{\textit{Keywords }} #1}

\input{include}
\title{Chance constrained problems: a bilevel convex optimization perspective}

\author{
Yassine Laguel$^{1}$ \hspace{2em}
J\'{e}r\^{o}me Malick$^{1}$  \hspace{2em}
Wim Van Ackooij$^{2}$ \hspace{2em}\\
\small{$^{1}$Univ. Grenoble Alpes, CNRS, LJK, 38000 Grenoble, France}  \\
\small{$^{2}$EDF R\&D, Saclay, France}}
\date{\vspace{-2em}}

\begin{document}
\maketitle



\begin{abstract}
Chance constraints are a valuable tool for the design of safe decisions in uncertain environments; they are used to model satisfaction of a constraint with a target probability.
However, because of possible non-convexity and non-smoothness, optimizing over a chance constrained set is challenging. In this paper, we
establish an exact reformulation of 
chance constrained problems as a bilevel problems with convex lower-levels. We then derive a tractable penalty approach, where the penalized objective is a difference-of-convex function that we minimize with a suitable bundle algorithm. We release an easy-to-use open-source python toolbox implementing the approach, with a special emphasis on fast computational subroutines.
\end{abstract}


\section{Introduction}\label{intro}
Chance constraints appear as a versatile way to model the exposure to uncertainty in optimization. Introduced in \cite{charnes1959chance}, they have been used in many applications, such as in energy management\;\cite{prekopa1978flood,vanAckooij_Henrion_Moller_Zorgati_2011b}, in telecommunications\;\cite{medova1998chance} or for reinforcement learning\;\cite{chow2017risk}, to name of few of them. We refer to the seminal paper\;\cite{Prekopa_1995}, the book chapter\;\cite{Dentcheva_2009} for introduction to the theory and to the recent article\;\cite{vanAckooij_2020} for a discussion covering recent developments.

In this paper, we consider a general chance-constrained optimization problem of the following form. For a fixed safety probability level $p \in [0,1)$, we write:
\begin{equation}\label{eq:general_prob}
\begin{split}
    &\!\!\min_{x\in \setX}\;\; f(\x)\\
    &\text{s.t.} \;\; \probability[g(\x, \xi) \leq 0] \geq p,
\end{split}
\end{equation}
where $f\colon \Rd \rightarrow \R$ and $g\colon \Rd \times \Rm \rightarrow \R$ are two given functions, {$\xi$ is a random vector valued in $\Rm$} and  $\setX \subset \Rd$ is a (deterministic) closed constraint set.

We consider the case of underlying convexity: We assume that $f$ and $g$ are convex (with respect to $\x$). For our practical developments, we also assume that we have first-order oracles for $f$ and $g$ and that the $\setX$ is a box constraint on the decision variable $x$.
Even with underlying convexity, modeling uncertainty may make the chance constraint feasible set non-convex (see e.g.\;\cite{Henrion_Strugarek_2008} for discussion on possible convexity when $p$ is close to $1$).
Though solving (non-convex) chance-constrained problems is difficult, several computational methods have been proposed, regardless of any considerations of convexity and smoothness, and under various assumptions on uncertainty. Let us mention: sample average approximation \cite{pagnoncelli2009sample,Luedtke_Ahmed_2008}, scenario approximation \cite{calafiore2006scenario}, convex approximation \cite{nemirovski2006convex}, or $p$-efficient points \cite{dentcheva2000concavity}; see e.g.\;\cite{vanAckooij_2020} for an overview.

In this paper, we propose an original approach for solving chance-constrained optimization problems.
First, we present an exact reformulation of (nonconvex) chance-constrained problems as (convex) bilevel optimization problems. This reformulation is simple and natural, involving superquantiles (also called conditional vale-at-risk), a risk measure studied by T.\;Rockafellar and his co-authors; see e.g.,\;the tutorial\;\cite{rockafellar2013superquantiles}.
Second, exploiting this bilevel reformulation, we propose a general algorithm for solving chance-constrained problems, and we release an open-source python toolbox implementing it. In the case where we make no assumption on the underlying uncertainty and have only samples of $\xi$, we propose and analyse a double penalization method, leading to an unconstrained single level DC (Difference of Convex) programs.
Our approach enables to deal with a fairly large sample of data-points in comparison with state-of-the-art methods based on mixed-integer reformulations, e.g.~\cite{ahmed2008solving}.
Thus our work mixes a variety of techniques coming from different subdomains of optimization: penalization, error bounds, DC programming,
bundle algorithm, Nesterov's smoothing; relevant references are given along the discussion.

This paper is structured as follows. In Section\;\ref{sec:bilevel}, we leverage the known link with (super)quantiles and chance-constraint to establish a novel bilevel reformulation of general chance constrained problems. In Section\;\ref{sec:penalty}, we propose and analyse a penalty approach revealing the underlying DC structure. In Section\;\ref{sec:bundle}, we discuss implementation of this approach in our publicly available toolbox. In section\;\ref{sec:experiments}, we provide illustrative numerical experiments, as a proof of concept, showing the interest of the method. Technical details on secondary theoretical points and on implementation issues are postponed to appendices.

\section{Chance constrained problems seen as bilevel problems}\label{sec:bilevel}

In this section, we derive the reformulation of a
chance constraint as a bilevel program wherein both the upper and
lower level problems, when taken individually, are convex.
We first recall in Section\;\ref{sec:recalls} useful definitions.
Our terminology and notations closely follow those of~\cite{rockafellar2013superquantiles}.

\subsection{Basics: cumulative distributions functions, quantiles, and superquantiles}\label{sec:recalls}

In what follows, we consider a probability space
and integrable real random variables.
Given a random variable $\rv$, its \textit{cumulative distribution function}, denoted by $\cdf{\rv}:\R \rightarrow [0,1]$, is defined as:
\begin{equation}\label{eq:def_cdf}
    \cdf{\rv}(t) := \probability[\rv \leq t] \quad \forall t \in \R.
\end{equation}
The cumulative distribution function is known to be both
non-decreasing and right-continuous. Its jumps occur exactly at
the atoms of $\rv$,  {that is} the values $t \in \R$ at
which $\probability[\rv = t] > 0$. These properties {enable one to define} the \textit{quantile function} $p
\mapsto \pquantile{\rv}$ as the following generalized inverse:
\begin{equation}\label{eq:def_quantile}
    \pquantile{\rv} = \inf\{t \in \R : \; \cdf{\rv}(t) \geq p\}, \quad \forall p \in [0,1).
\end{equation}
If $\rv$ is assumed to belong to $\integrable$, we can
additionally define for any $p \in [0,1)$ its
$p$-\textit{superquantile}, $\psuperquantile{\rv}$ as follows:
\begin{equation}\label{eq:def_superquantile}
    \psuperquantile{\rv} = \frac{1}{1-p} \int_{p' = p}^1 Q_{p'}(\rv) dp'.
\end{equation}


For a given random variable $\rv$, as a consequence of~\cite[Th.\;2]{rockafellar2014random}, one can recover from the cumulative distribution function $\cdf{\rv}$ both the $p$-quantile and the $p$-superquantile functions as functions of $p$ and reciprocally, knowing either the $p$-quantile or the $p$-superquantile for all $p \in [0,1]$ suffices to recover $\cdf{\rv}$.

From a statistical
viewpoint, these three notions are also equally consistent \cite[Th.\;4]{rockafellar2013superquantiles} in the sense that convergence in
distribution for a sequence of random variables $(\rv_n)_{n \geq
0}$ is equivalent to the  {pointwise} convergence of the
two sequences of functions $p \mapsto \pquantile{\rv_n}$, $p
\mapsto \psuperquantile{\rv_n}$. This result is particularly
relevant when the distributions are observed through data
sampling. We can use the empirical cumulative distribution
functions, quantiles and super-quantiles all while upholding
asymptotic convergence as the sample size grows.

From an optimization point of view though, these three objects are
very different. In contrast with the others,
the superquantile has several good
properties (including convexity~\cite{ben2007old,follmer2002convex,ruszczynski2006optimization}), useful with respect to numerical computation and optimization.
In our developments, we use the following key result
\cite[Th. 1]{rockafellar2000optimization} linking quantiles and
superquantile through a 
one-dimensional problem.

\begin{lemma}\label{lem:rock_cvar}
For an integrable random variable $\rv$ and a
probability level $p$, the superquantile $\psuperquantile{\rv)}$ and quantile
$\pquantile{\rv}$ are respectively the optimal value and the
optimal solution of the convex one-dimensional problem
\begin{equation}\label{eq:link_theorem}
         \inf_{\eta \in \R} ~\eta + \frac{1}{1-p} \expectation[\max(\rv -\eta, 0)].\\
\end{equation}
\end{lemma}


\subsection{Reformulation as a bilevel problems}

By definition, the chance constraint in \eqref{eq:general_prob}
involves the cumulative distribution function:
we have for any fixed $\x \in \Rd$,
$\probability[g(\x, \xi) \leq 0] \geq p \Leftrightarrow \cdf{g(\x,
\xi)}(0) \geq p$.
%
Following the discussion of the previous section, we easily rewrite this constraint using quantiles, as formalized in
the next lemma.
\begin{lemma}\label{lem:proba_to_quantiles}
    For any $\x \in \Rd$ and $p \in [0,1)$, we have:
\begin{equation*}
    \probability[g(\x, \xi) \leq 0] \geq p \iff
    \pquantile{g(\x, \xi)} \leq 0.
\end{equation*}
\end{lemma}

\begin{proof}
By definition of the quantile and continuity on the right of the
cumulative distribution function, we always have $p \leq
\probability[g(\x, \xi) \leq \pquantile{g(\x, \xi)}]$. Thus, since
cumulative distribution functions are increasing, if
$\pquantile{g(\x, \xi)} \leq 0$, then $\probability[g(\x, \xi)
\leq \pquantile{g(\x, \xi)}] \leq \probability[g(\x, \xi) \leq 0]$
which implies that  {$\probability[g(\x, \xi) \leq 0] \geq
p$}.

Conversely, since $\pquantile{g(\x, \xi)}$ is the infimum of
$\{t\in \R \;: \probability[X \leq t] \geq p\}$, if
$\pquantile{g(\x, \xi)} > 0$, then necessarily we have
$\probability[g(\x, \xi) \leq 0] < p$.
\end{proof}

Together with\;\eqref{eq:link_theorem}, we obtain from the previous easy lemma a bilevel formulation of the general chance-constrained problem \eqref{eq:general_reformulation}. The idea is simple: introducing an auxiliary variable $\eta\in \Rd$ to recast the potentially
non-convex chance constraint of \eqref{eq:general_prob} as two
constraints, a simple bound constraint and a difficult optimality constraint, forming a lower subproblem.
Introducing the lower objective function $G: \setX
\times \R \rightarrow \R$
\begin{equation}\label{eq:G}
    G(\x, s) = s + \frac{1}{1-p} \expectation[\max(g(\x,\xi) - s, 0)],
\end{equation}
we have the following exact reformulation of chance-constrained problems.
\begin{theorem}\label{thm:bilevel_reformulation}
    Problem \eqref{eq:general_prob} is equivalent to the bilevel problem:
    \begin{equation}\label{eq:general_reformulation}
    \left\{
    \displaystyle\begin{array}{ll}
        \min_{\x \in \setX, \eta \in \R} & f(\x) \\
        \text{s.t.} & \eta \leq 0 \\
         & \eta \in S(\x) = \argmin_{s \in \R} G(\x,s).\\
    \end{array}
\right.
\end{equation}
More precisely, if $\x^\star$ is an optimal solution of \eqref{eq:general_prob},
then $(\x^\star, \pquantile{g(\x^\star, \xi)})$ is an optimal
solution of the above bilevel problem, and conversely.
\end{theorem}

\begin{proof}
It is clear with Lemma \ref{lem:proba_to_quantiles} that problem
\eqref{eq:general_prob} is equivalent to
\begin{equation}\label{eq:pb_quantiles_non_convex}
\left\{
\begin{array}{ll}
        \displaystyle\min_{\x \in \setX, \eta \in \R} & f(\x)  \\
        \text{s.t.} & \eta \leq 0 \\
         & \eta = \pquantile{g(\x, \xi)}
\end{array}
\right. .
\end{equation}
By Lemma~\ref{lem:rock_cvar}, $\pquantile{g(\x, \xi)} \in S(\x)$ for any $\x \in
\R^d$. Hence, any solution $(x,\eta)$ of
$\eqref{eq:pb_quantiles_non_convex}$ is feasible for
\eqref{eq:general_reformulation}. Conversely, any solution
$(\x^{\star}, \eta^{\star})$ of \eqref{eq:general_reformulation}
satisfies: $\pquantile{g({\x}^{\star}, \xi)} \leq \eta^{\star}
\leq 0$ which implies that $(\x^\star, \pquantile{g({\x}^{\star},
\xi)}$ is a feasible point of \eqref{eq:pb_quantiles_non_convex}.
Since both problems have the same objective, they are equivalent.
\end{proof}

The first constraint $\eta \leq 0$ is an easy
one-dimensional bound constraint which does not involve the decision
variable $\x$. The second constraint, which constitutes the lower
level problem is more difficult; when this constraint is
satisfied, $\eta$ is exactly the $p$-quantile of $g(\x, \xi)$.
We readily see the joint convexity of the objective function of
the lower level problem in \eqref{eq:general_reformulation} with
respect to $s$ and $\x$.

This bilevel reformulation is nice, natural and seemingly new;
we believe that it opens the door to new approaches for solving chance-constrained problems.
In the next section, we propose such an approach based on the reformulation.

\section{A double penalization scheme for chance constrained problems}\label{sec:penalty}


In this section, we explore one possibility offered by the bilevel formulation of chance-constrained problems, presented in the previous section. We propose a (double) penalization approach for solving the bilevel optimization problem,
with a different treatment of the two constraints: a basic penalization of the easy constraint together with
an exact penalization of the hard constraint formalized as the lower problem. We first derive in Section\;\ref{sec:analysis_vf},
some growth properties of the lower problem.
We show then in Section\;\ref{sec:exact} to what extent these
properties help to provide an exact penalization of the ``hard"
constraint. We finally present the double penalty scheme
in section \ref{sec:convergence_analysis}.

From the bilevel problem\;\eqref{eq:general_reformulation}, we derive the two following penalized problems, associated with two penalization parameters $\mu,
\lambda > 0$ and
\begin{equation}\label{eq:problem_pmu}
    (P_{\mu}) \quad \left\{
    \begin{array}{ll}
        \displaystyle\min_{(\x,\eta) \in \setX \times \R} & f(\x) + \mu \max(\eta, 0) \\
        \mbox{s.t.} & \eta \in \argmin_{s\in \R} G(\x, s)
    \end{array}
\right.
\end{equation}
\vspace*{-1ex}
and
\begin{equation}\label{eq:def_penalized_pb}
    \Plmu \quad \min_{(\x, \eta) \in \setX \times \R} f(\x) + \lambda \left(G(\x,\eta) - \min_{s \in \R} G(\x, s)\right) + \mu \max(\eta, 0).
\end{equation}
We consider a general data-driven situation where the uncertainty
$\xi$ is just known through a sample (or, said
alternatively, follows an equiprobable discrete distribution over
$n \in \N$ arbitrary values): we assume that there exists $\xi_1, \xi_2, \dots, \xi_n \in \Rm$ such that $\probability[\xi = \xi_i] = \frac{1}{n}$ for all $i \in \{1, \dots, n\}$. The set
$\indexset$ defined as
\begin{equation}\label{eq:def_index_set}
    \indexset = \left\{\frac{i}{n}, \quad i \in \{0,...,n-1\}
    \right\}
\end{equation}
plays a special role in our developments. In particular, we use the distance to~$\indexset$, denoted by $\dist{\indexset}(p)$, to define a key quantity appearing in the variational results of this section: we introduce
    \begin{equation}\label{eq:def_delta}
        \delta = \left\{
        \begin{array}{ll}
            \displaystyle\frac{1}{n(1-p)} &\mbox{ ~~if }p \in \indexset \\[2ex]
            \displaystyle\frac{\dist{\indexset}(p)}{(1-p)} &\mbox{ ~~otherwise,}\\
        \end{array}
        \right.
    \end{equation}
which depends implicitly on the number of samples $n$ and the fixed safety parameter $p$.

\subsection{Analysis of the Value function}\label{sec:analysis_vf}

In view of the forthcoming exact penalization,
we study here the value function $h : \setX \times \R \rightarrow \R$ defined, from $G$ in \eqref{eq:G}, as
\begin{equation}
h(\x,\eta) = G(\x, \eta) - \min_{s\in \R} G(\x,s).
\label{eq:herrfct}
\end{equation}
The next result relates $h$ to $\dist{S(\x)}(\cdot)$, the distance function to $S(\x)$, the solution set of the lower level problems in \eqref{eq:general_reformulation}. This is our main technical result, on which next propositions are based.

\begin{theorem}\label{uniform_parametric_error_bound}
Let $p\in [0,1)$ be fixed but arbitrary. The function
$h$ defined in\;\eqref{eq:herrfct} satisfies for any $(\x, \eta)
\in \setX \times \R$
\begin{equation*}
                h(x,\eta) \geq \delta\;\dist{S(\x)}(\eta)
                \qquad\text{with $\delta$ defined by \eqref{eq:def_delta}.}
\end{equation*}
\end{theorem}

\begin{proof}
Let us fix $\x \in \setX$ and denote by $q_p$ the $p$-quantile of
$g(\x, \xi)$. We first note that by the arguments in the proof of
Lemma \ref{lem:proba_to_quantiles}, we have
\begin{equation*}
    p \leq \probability[g(\x, \xi) \leq q_p],
\end{equation*}
with equality holding in the left inequality, if and only if
$p$ belongs to\;$\indexset$.

For any fixed but arbitrary $\eta \in \R$, we have the
following identity:
\begin{equation*}
\begin{split}
    h(\x,\eta) &= \eta + \frac{1}{1-p} \mathbb{E}[\max(g(\x, \xi) - \eta, 0)] -   \left( q_p + \frac{1}{1-p} \mathbb{E}[\max(g(\x, \xi) - q_p, 0)] \right) \\
    &= (\eta - q_p) + \frac{1}{1-p} \mathbb{E} \left[\max(g(\x,\xi), \eta) - \eta - (\max(g(\x,\xi), q_p) - q_p) \right] \\
    &= (\eta - q_p) (1 - \frac{1}{1-p}) + \frac{1}{1-p} \mathbb{E} \left[\max(g(\x,\xi), \eta) - \max(g(\x,\xi), q_p) \right].\\
\end{split}{}
\end{equation*}

Now, by employing a case distinction on the location of $\eta$
with respect to $q_p$, we will derive the desired inequalities.
Let us first consider that $\eta > q_p$, then we have:
\begin{equation*}
    \begin{split}
        h(\x,\eta) &=  (\eta - q_p) (1 - \frac{1}{1-p}) + \frac{1}{1-p} \mathbb{E} \left[ (\eta - g(\x, \xi)) \one_{q_p< g(\x,\xi) \leq \eta} + (\eta - q_p) \one_{g(\x,\xi) \leq q_p} \right]\\
        &= (\eta - q_p) (1 - \frac{1}{1-p} + \frac{1}{1-p} \probability[g(\x, \xi) \leq q_p]) + \frac{1}{1-p} \expectation[(\eta - g(\x,\xi)) \one_{q_p < g(\x,\xi) \leq \eta}] 
    \end{split}
\end{equation*}
which finally gives:
\begin{equation}\label{eq:error_first_case}
        h(\x,\eta)= \frac{(\eta - q_p)}{1-p} \!\left(\probability[g(\x, \xi) \leq q_p] - p + \expectation \left [\frac{\eta - g(\x, \xi)}{\eta - q} \one_{q_p < g(\x,\xi) \leq \eta} \right ]
        \right)\!.
\end{equation}
Now if $\probability[g(\x, \xi) \leq q_p] > p$, this implying that
$p \notin \indexset$, then by non-negativity of the expectation
term above, we have:
\begin{equation*}
h(\x, \eta) \geq (\eta - q_p) \frac{1}{1-p} \left
(\probability[g(\x, \xi) \leq q_p] - p \right ) \geq
\frac{\dist{\indexset}(p)}{1-p}\; \dist{S(\x)}(\eta).
\end{equation*}
Here we use that clearly $\abs{\eta - q_p} \geq
\dist{S(\x)}(\eta)$, since $q_p \in S(x)$ as already recalled.
Furthermore $p \leq \probability[g(\x, \xi) \leq q_p] < p + \frac
1n$ by Lemma \ref{lem:proba_to_quantiles} and by definition. We also observe that
$\probability[g(\x, \xi) \leq q_p] \in \indexset$, so that altogether we have:
\begin{align}
0 &\leq (\probability[g(\x, \xi) \leq q_p] - p) \notag\\
(\probability[g(\x, \xi) \leq q_p] - p) &\geq \dist{\indexset}(p) \notag\\
\dist{\indexset}(p) &\leq \tfrac{1}{n}\;.\notag
\end{align}
If to the contrary $p \in \indexset$ which implies that
$\probability[g(\x, \xi) \leq q_p] = p$, then $h(x,\eta) =
\frac{1}{1-p}\expectation[(\eta - g(\x, \xi)) \one_{q_p <
g(\x,\xi) \leq \eta}]$. We let $q_p^{+}$ be the successor
quantile, i.e.,
\begin{equation*}
q_p^{+} = \inf\{t \geq \R :\; \probability[g(\x, \xi)\leq t] > p
\}.
\end{equation*}
Since $p \leq \frac{n-1}{n}$, it follows that $q_p^+ < \infty$ and
$q_p^+ > q_p$. Now if $\eta \in (q_p, q_p^{+})$, we have $h(x,
\eta) = 0$. If $\eta \geq q_p^{+}$, then
\begin{equation*}
\begin{split}
h(x,\eta)   &= \frac{1}{1-p} \expectation \left [(\eta - g(\x, \xi)) \one_{q_p^{+} \leq g(\x,\xi) \leq \eta} \right ] \\
                            & \geq (\eta - q_p^{+}) \frac{\probability[g(\x,\xi) = q_p^{+}]}{1-p} \geq \frac{1}{n(1-p)} \dist{S(\x)}(\eta) \\
                            & \geq \frac{\dist{\indexset}(p)}{(1-p)}
                            \dist{S(\x)}(\eta),
\end{split}
\end{equation*}
where the last inequality results from our earlier estimates.

The second case to consider involves the situation $\eta < q_p$.
Here, we have:
\begin{equation*}
    \begin{split}
        h(\x, \eta) &=  (\eta - q_p) (1 - \frac{1}{1-p}) + \frac{1}{1-p} \expectation \left[(g(\x, \xi) - q_p) \one_{\eta< g(\x,\xi) \leq q_p} + (\eta - q_p) \one_{g(\x,\xi) \leq \eta} \right]\\
        &= (\eta - q_p)\left (1 - \frac{1}{1-p} + \frac{1}{1-p} \probability[g(\x, \xi) \leq \eta] \right ) + \frac{1}{1-p} \expectation \left[(g(\x, \xi) - q_p) \one_{\eta< g(\x,\xi) \leq q_p}\right]
    \end{split}
\end{equation*}
which leads us to
\begin{equation}\label{eq:error_second_case}
        h(\x, \eta) =  \frac{(q_p - \eta)}{1-p} \!\left(\!p - \probability[g(\x, \xi) \leq \eta] - \expectation \!\left[\frac{q_p - g(\x, \xi)}{q_p - \eta}\! \one_{\eta< g(\x,\xi) \leq q_p}\right]\!\right)\!.
\end{equation}
Now if $\probability[g(\x, \xi) \leq q_p] > p$, this implies that
$p \notin \indexset$, then let us define the antecessor quantile
$q_p^-$ as
\begin{equation*}
q_p^{-} = \max\left\{\sup\{t \geq \R :\; \probability[g(\x, \xi)\leq t]
< p \}, \min\{g(x,\xi_i)\}_{i=1}^n - 1 \right \}.
\end{equation*}
We can first observe that since $p \notin \indexset$, we can
entail $p > 0$, hence $q_p^- > -\infty$ is well defined. For any
$\eta \in (q_p^-, q_p)$, it follows that $\expectation
\left[\frac{q_p - g(\x, \xi)}{q_p - \eta} \one_{\eta< g(\x,\xi)
\leq q_p}\right] = 0$. We may thus consider that $\eta \leq
q_p^-$, in which case we have:
\begin{equation*}
\begin{split}
h(\x, \eta) &= (q_p - \eta) \frac{1}{1-p} \left(p -
\probability[g(\x, \xi) \leq \eta] - \expectation \left[\frac{q_p
- g(\x, \xi)}{q_p - \eta} \one_{\eta< g(\x,\xi) \leq
q_p^-}\right]\right) \\
&\geq (q_p - \eta) \frac{1}{1-p} \left(p - \probability[g(\x, \xi) \leq \eta] - \expectation \left[\one_{\eta< g(\x,\xi) \leq q_p^-}\right]\right) \\
            &\geq (q_p - \eta) \frac{1}{1-p}  \left(p - \probability[g(\x, \xi) \leq q_p^-]\right) \\
            &\geq \frac{1}{n(1-p)}\; \dist{S(\x)} (\eta) \geq \frac{\dist{\indexset}(p)}{1-p}\; \dist{S(\x)}(\eta),
\end{split}
\end{equation*}
where we have used that $\frac{q_p - g(x,\xi)}{q_p - \eta} \leq 1$
on $\one_{\eta < g(x,\xi) \leq q_p}$.

If to the contrary, $p\in \indexset$, which implies
$\probability[g(\x, \xi) \leq q_p] = p$, recalling the identity $p
- \probability[g(\x, \xi) \leq \eta] = \expectation[ \one_{\eta <
g(x,\xi) < q_p} ] + \probability[g(\x, \xi) = q_p]$, we obtain:
        \begin{equation*}
            \begin{split}
                h(\x,\eta) &= (q_p - \eta) \frac{1}{1-p} \left( \expectation \left[ 1 - \frac{(q_p - g(\x, \xi) )}{q_p - \eta} \one_{\eta< g(\x,\xi) < q_p}\right] + \probability[g(\x, \xi) = q_p] \right)\\
                &\geq (q_p - \eta) \frac{\probability[g(\x, \xi) = q_p]}{1-p}\\
&\geq \frac 1n \frac{1}{1-p}\dist{S(\x)}(\eta) \geq
\frac{\dist{\indexset}(p)}{1-p}\; \dist{S(\x)}(\eta),
            \end{split}
        \end{equation*}
where we have used that $\probability[g(\x, \xi) = q_p] = \frac
1n$. The last case $\eta = q_p$, gives by construction $\eta \in
S(x)$, i.e., $\dist{S(\x)}(\eta) = 0$ and clearly $h(x,\eta)=0$ so
that the desired inequality holds.
\end{proof}

Following the terminology of \cite{ye1997exact}, this theorem shows that $h$ is
a uniform parametric error bound.
We note that the quality of this bound
is altered by the
number $n$ of data points considered. This drawback actually
passes to the limit in the sense that $(x,\eta) \mapsto h_x(\eta)$
fails to be a uniform parametric error bound when $\xi$ follows a
continuous distribution; this is an interesting but secondary result that we prove in Appendix\;\ref{app:continuous_case}.

\subsection{An exact penalization for the hard constraint}\label{sec:exact}
We show here that $\Plmu$ is an exact penalization of
$(P_{\mu})$, when $\lambda$ is large enough. The proof of this result follows usual rationale (see e.g.,\;\cite[Prop.\;2.4.3]{clarke1990optimization}); the main technicality is the sharp growth of $h$ established in Theorem\;\ref{uniform_parametric_error_bound}.

\begin{proposition}\label{prop:exact_penalization}
Let $\mu > 0$ be given and assume that there is a solution to $(P_{\mu})$ defined in \eqref{eq:problem_pmu}. Then for any $\lambda > \mu/\delta$ with $\delta$ defined in\;\eqref{eq:def_delta}, the solution set of $(P_{\mu})$ coincides with the one of $(P_{\lambda, \mu})$ defined in \eqref{eq:def_penalized_pb}.
\end{proposition}

\begin{proof}
Take $\mu>0$, define $\lambda_\mu= \mu/\delta$, and take $\lambda > \lambda_{\mu}$ arbitrary but fixed.
Let us first take a solution $({\x}^{\star}, \eta^{\star}) \in
\setX \times \R$ of $(P_{\mu})$ and show by contradiction that it is
also a solution of $(P_{\lambda, \mu})$. Indeed, to the contrary,
assume there exists  some $\varepsilon > 0$ and $({\x}', \eta') \in \setX
\times \R$ such that:
    \begin{equation*}
        f({\x}') + \mu \max(0, \eta') + \lambda h_{{\x}'}(\eta') \leq f({\x}^{\star}) + \mu \max(0, \eta^{\star}) + \lambda\; h_{{\x}^{\star}}(\eta^{\star}) -
        \varepsilon.
    \end{equation*}

Let then $\eta_p' \in S({\x}')$ be such that : $|\eta_p' - \eta'|
\leq \dist{S({\x}')}(\eta') + \frac{\varepsilon}{2 \mu}$. Then the
point $({\x}', \eta_p')$ is a feasible for $P_{\mu}$ (recall
$\eta_p' \in S({\x}')$) and since $\eta \mapsto \mu \max(0, \eta)$
is $\mu$-Lipschitz, we first have
    \begin{equation*}
    \begin{split}
        f({\x}') + \mu \max(0, \eta_p') &\leq   f({\x}') + \mu \max(\eta',0) + \mu |\eta_p' - \eta'|\\
        &\leq f({\x}') + \mu \max(\eta',0) + \mu \left( \dist{S({\x}')}(\eta') + \frac{\varepsilon}{2 \mu} \right).
    \end{split}{}
    \end{equation*}
Using Theorem \ref{uniform_parametric_error_bound}, we then have
    \begin{equation*}
    \begin{split}
        f({\x}') + \mu \max(0, \eta_p') &\leq
        f({\x}') + \mu \max(\eta',0) + \mu \left(\frac{1}{\delta} h(\x',\eta') + \frac{\varepsilon}{2 \mu}\right)\\
        &\leq  f({\x}') + \mu \max(\eta',0) + \lambda_{\mu}\; h({\x}',\eta') + \frac{\varepsilon}{2}\\
        &\leq  f({\x}^{\star}) + \mu \max(\eta^{\star},0) - \frac{\varepsilon}{2}\\
    \end{split}{}
    \end{equation*}
which gives the contradiction. Hence any solution of $(P_\mu)$ is
also a solution to problem $(P_{\lambda, \mu})$.

Let now $(\bar{\x}, \bar{\eta})$ be
a solution of $(P_{\lambda, \mu})$ and let us show that it is actually a solution for $P_{\mu}$. Let again $({\x}^{\star}, \eta^{\star})$ be an arbitrary solution of $(P_{\mu})$. We first note that that a result of optimality of $(\bar x, \bar \eta)$ for $(P_{\lambda, \mu})$, we have:
\[
    f(\bar{\x}) + \mu \max(0,\bar \eta) + \lambda \underbrace{h(\bar \x, \bar \eta)}_{\geq 0} \leq f({\x}^\star) + \mu \max(0,\eta^\star) + \lambda \underbrace{h({\x}^\star, \eta^\star)}_{=0},
\]
which by positivity of the function $h$ and feasibility for $(P_{\mu})$, i.e., $h(x^\star, \eta^\star)=0$ of $({\x}^{\star}, \eta^{\star})$ yields:
\[
    f(\bar{\x}) + \mu \max(0,\bar \eta) \leq f({\x}^\star) + \mu \max(0, \eta^\star).
\]

    It remains to show that $(\bar \x, \bar \eta)$ is a feasible point for $(P_\mu)$. By the first point, $({\x}^{\star}, \eta^{\star})$ is both a solution of $(P_{\lambda, \mu})$ and $(P_{\frac{\lambda + \lambda_{\mu}}{2}, \mu})$. Hence, we have:
    \begin{equation*}
    \begin{split}
        f(\bar{\x}) + \mu\max(0, \bar{\eta}) + \lambda h(\bar{\x}, \bar{\eta}) &\leq  f({\x}^{\star}) + \mu \max(0, {\eta}^{\star}) \\
        &= f({\x}^{\star}) + \mu \max(0, {\eta}^{\star}) + \frac{\lambda + \lambda_\mu}{2} h(x^\star,\eta^\star) \\
        &\leq f(\bar{\x}) + \mu \max(0, \bar{\eta}) + \frac{\lambda + \lambda_\mu}{2} h(\bar{\x},\bar{\eta})\\
    \end{split}
    \end{equation*}
    But since $\lambda > \lambda_\mu$ we necessarily have: $h(\bar{\x},\bar{\eta}) = 0$ which implies by the properties of the value function that $(\bar{\x}, \bar{\eta})$ is a feasible point for $(P_{\mu})$.
\end{proof}

\subsection{Double penalization scheme}\label{sec:convergence_analysis}

From the previous results, we get that solving the sequence of penalized problems gives approximations of the solution of the initial problem. We formalize this 
in the next proposition suited for our context of double penalization.
The proof of this result follows standard arguments; see e.g.\;\cite[Ch.\;13.1]{luenberger1984linear}.

\begin{proposition}\label{prop:convergence_of_penalty}
Assume that Problem\;\eqref{eq:general_reformulation} has a non-empty feasible set.
Let $(\mu_k)_{k \geq 0}$ be an increasing sequence such that $\mu_k \shortarrow{1} \infty$, and
$(\lambda_k)_{k \geq 0}$ be taken such that $\lambda_k > \frac{\mu_k}{\delta}$ with $\delta$ as defined in \eqref{eq:def_delta}.
If, for all $k$, there exists a solution of\;$(P_{\lambda_k, \mu_k})$ (denoted by $(x_k, \eta_k)$),
then any cluster point of the sequence $({\x}_k, \eta_k)$ is an optimal solution of~\eqref{eq:general_prob}.
\end{proposition}

\begin{proof}

The fact that $(x_k,\eta_k)$ is an optimal solution of\;$(P_{\lambda_k,\mu_k})$ implies that
\begin{align}\label{seq:1}
f({\x}_k) + \mu_k \max(0, \eta_k) &+ \lambda_k h(x_k,\eta_k)\\
&\leq f(x_{k+1}) + \mu_k \max(0, \eta_{k+1}) + \lambda_k
h(x_{k+1},\eta_{k+1}) \notag
\end{align}
Similarly for $(x_{k+1},\eta_{k+1})$, we get
\begin{align*}
f(x_{k+1})  + \mu_{k+1} \max(0, \eta_{k+1}) &+ \lambda_{k+1} h(x_{k+1},\eta_{k+1}) \notag \\
& \leq f(x_k) + \mu_{k+1} \max(0, \eta_k) + \lambda_{k+1}
h(x_k,\eta_k).
\end{align*}
By Proposition\;\ref{prop:exact_penalization}, $\eta_k$ (resp. $\eta_{k+1}$) is feasible for $(P_{\mu_k})$ (resp. $(P_{\mu_{k+1}})$); in other words, we have $h(\x_k,\eta_k) = h(\x_{k+1},\eta_{k+1}) = 0$. Hence summing up these two inequalities yields
\begin{equation*}
        \max(\eta_k, 0) \geq \max(\eta_{k+1}, 0).
\end{equation*}
Using this last inequality with \eqref{seq:1} gives:
    \begin{equation*}
        \begin{split}
            f({\x}_k) - f({\x}_{k+1}) \leq \mu_k \left (\max(\eta_{k+1}, 0) - \max(\eta_{k}, 0)\right) \leq
            0,
        \end{split}
    \end{equation*}
and as a consequence the sequence $\{f({\x}_k)\}_{k \geq 0}$
increases. Let $(\x', \eta')$ be an arbitrary feasible solution
for $(P)$. By definition of the sequence $({\x}_k, \eta_k)$, for
any $k \in \N$, we have:
    \begin{equation}\label{eq:increase_function_values}
        f({\x}_k) \leq f({\x}_k) + \mu_k \max(\eta_k, 0) \leq f({\x}') + \mu_k \max(\eta', 0) \leq
        f({\x}').
    \end{equation}
Therefore for any cluster point $(\bar{\x}, \bar{\eta})$ of the sequence $\{({\x}_k, \eta_k)\}_{k \geq 0}$, we have $f(\bar{\x}) \leq f(\x')$. In order to show that $(\bar{\x}, \bar{\eta})$ is a solution of~\eqref{eq:general_reformulation}, it remains to show its feasibility.
With the right hand side inequality of~\eqref{eq:increase_function_values}, we obtain
    \begin{equation*}
        \max(\eta_k, 0) \leq \frac{f(\x') - f({\x}_k)}{\mu_k} \leq \frac{f(\x') - f(\x_0)}{\mu_k} \xrightarrow[k \rightarrow \infty]{}
        0,
    \end{equation*}
so that we may deduce that, $\bar{\eta} \leq 0$. Moreover, continuity of $h$ ensures that $h(\bar x, \bar \eta)=0$ which completes the
proof.
\end{proof}

In words, cluster points of a
sequence of solutions obtained as $\mu$ grows to\;$+\infty$ are
feasible solutions of the initial chance-constrained problem. In practice though, we have observed that taking a fixed
$\mu$ is enough for reaching good approximations of the solution with increasing $\lambda$'s; see in particular the numerical experiments of Section~\ref{sec:experiments}. In the next section, we discuss further the practical implementation of the conceptual double penalization scheme.

\section{Double penalization in practice}
\label{sec:bundle}

In this section, we propose a practical version of the double penalization scheme for solving chance-constrained optimization problems.
First, we present in Section\;\ref{sec:bundle_algorithm} how to tackle the inner penalized problem $(P_{\lambda, \mu})$ by leveraging its difference-of-convex (DC) structure.
Then we quickly describe, in Section\;\ref{sec:taco}, the python toolbox that we release, implementing this bundle algorithm and efficient oracles within the double penalization method.

\subsection{Solving penalized problems by a bundle algorithm}
\label{sec:bundle_algorithm}

We discuss here an algorithm for solving $(P_{\lambda, \mu})$ by revealing the DC structure of the objective function. Notice indeed that, introducing the two convex functions
\[
\varphi_1(\x, \eta) = f(\x) + \lambda G(\x, \eta) + \mu
\max(\eta, 0) \quad\text{and}\quad
\varphi_2(\x, \eta) = \lambda \min_{s\in \R}
G(\x, s)\]
we can write $(P_{\lambda, \mu})$ as the DC problem
\begin{equation}\label{eq:dc}
\min_{(\x,\eta) \in \setX\times\R} \varphi(\x, \eta) = \varphi_1(\x, \eta) - \varphi_2(\x, \eta).
\end{equation}
We then propose to solve this problem by the bundle
algorithm of\;\cite{de2019proximal}, which showed to be a method of choice for DC problems.
This bundle algorithm interacts with first-order oracles for $\varphi_1$ and $\varphi_2$; in our situation, there exist computational procedures to compute subgradients of $\varphi_1$ and $\varphi_2$ from output of oracles of $f$ and $g$, as formalized in the next proposition. The proof of this proposition is deferred to Appendix~\ref{app:taco_details}. Note that at the price of more heavy expressions, we could derive the whole subdifferential.

\begin{proposition}\label{prop:oracle_computation}
Let $(x, \eta) \in \setX \times \R$ be fixed. Let $s_f$ be a subgradient of f at $x$ and ${s_g}_1, \dots, {s_g}_n$ be respective subgradients of $g(\cdot, \xi_1),\dots, g(\cdot, \xi_n)$ at $x$. For a given $t\in \R$, denote by $I_{>t}$ the set of indices such that $g(x, \xi_i) > t$ and by $I_{=t}$ the set of indices such that $g(x, \xi_i) = t$. Let finally $\alpha = \frac{\probability[g(x,\xi) \leq Q_p(g(x, \xi)] - p}{\#(I_{=Q_p(g(x, \xi))})}$. Then, $s_{\varphi_1}$ and $s_{\varphi_2}$ defined as:
\begin{equation*}
\begin{split}
    s_{\varphi_1} = &\left(s_f + \frac{\lambda}{n(1-p)} \sum_{i \in I_{>\eta}}^n {s_g}_i ~,~ 1 + \mu \mathds{1}_{\eta > 0}  - \lambda \frac{\#(I_{>\eta})}{n(1-p)}  \right)\\
\end{split}
\end{equation*}
\begin{equation*}
    s_{\varphi_2} = \left(\frac{\lambda}{n(1-p)} \left(\sum_{i \in I_{>Q_{p}(g(x,\xi))}} \!\!{s_g}_i + \alpha \!\!\sum_{i \in I_{=Q_{p}(g(x,\xi))}} \!\! {s_g}_i \right)~,~ 0 \right)  \\
\end{equation*}
are respectively subgradients of $\varphi_1$ and $\varphi_2$ at $(x, \eta)$.
\end{proposition}

Notice now that the convergence result for the bundle algorithm \cite[Th.\;1]{de2019proximal} guarantees convergence towards a point $\bar u = (\bar{x},\bar \eta$) satisfying
\begin{equation}\label{eq:def_critical}
    \partial \varphi_2(\bar u) \cap \partial \varphi_1(\bar u) \neq
    \emptyset,
\end{equation}
which is a weak notion of criticality.
Thus, we propose to furthermore replace  $\varphi_2$ in \eqref{eq:dc} by a smooth approximation of it, denoted by $\widetilde{\varphi}_2$. The reason is that the bundle method minimizing $\widetilde{\varphi} = \varphi_1 - \widetilde{\varphi}_2$ then reaches a Clarke-stationary point:
indeed,
\eqref{eq:def_critical} reads $\nabla \widetilde{\varphi}_2(\bar u) \subset \partial \varphi_1(\bar u)$, which gives $ 0 \in \partial\varphi(\bar u) =  \partial \varphi_1(\bar u) + \nabla \widetilde{\varphi}_2(\bar u)$, i.e.,\;that $\bar u$ is Clarke-stationary (for the smoothed problem).
To smooth $\varphi_2$, we use the efficient smoothing procedure of~\cite{laguel-etal:spqr:mlsp2020} for superquantile-based functions (implementing the Nesterov's smoothing technique~\cite{nesterov2005smooth}).
More precisely, \cite[Prop.\;2.2]{laguel-etal:spqr:mlsp2020}
reads as follows.
\begin{proposition}
Assume that $g$ is differentiable. For a smoothing parameter $\rho > 0$, the function 
\begin{equation}\label{eq:smoothin_pb}
    \widetilde{\varphi}_2(x, \eta) = \lambda \sup_{\substack{0\leq q_i \leq \frac{1}{n(1-p)} \\ q_1 + \dots + q_n = 1}} \sum_{i=1}^n \left\{q_i\; g(x, \xi_i) - \frac{\rho}{2} (q_i - \tfrac{1}{n})^2\right\}
\end{equation}
is a global approximation of $\varphi_2$, such that
$\widetilde{\varphi}_2(x, \eta) \leq \varphi_2(x, \eta) \leq \widetilde{\varphi}_2(x, \eta) + \frac{\lambda \rho}{2}$ for all $(x, \eta) \in {\R}^{d+1}$. Moreover, the function is differentiable and its gradient writes, with $S = ({s_g}_i)_{1 \leq q \leq n}$ the Jacobian of $x \mapsto (g(x, \xi_1),\dots, g(x, \xi_n))$, as
\[
    \nabla \widetilde{\varphi}_2 (x, \eta) = (\lambda\; S\; \Tilde{q}\;,\;0)
\]
where $\Tilde{q}$ is the (unique) optimal solution of~\eqref{eq:smoothin_pb}.
\end{proposition}

Note that the computation of $\widetilde{q}$ can be performed with fast computational procedures, as proposed in~\cite{laguel-etal:spqr:mlsp2020}.

\subsection{A python toolbox for chance constrained optimization}\label{sec:taco}


We release \texttt{TACO}, an open-source python toolbox for solving chance constrained optimization problems~\eqref{eq:general_prob}. The toolbox implements the penalization approach outlined in section~\ref{sec:penalty} together with the bundle method~\cite{de2019proximal} for the inner penalized subproblems. \texttt{TACO} routines rely on just-in-time compilation supported by Numba~\cite{10.1145/2833157.2833162}. The routines are optimized to provide fast performances on reasonably large datasets. Documentation is available at:
\begin{center}
\url{https://yassine-laguel.github.io/taco}
\end{center}
We provide here basic information on \texttt{TACO}; for further information, we refer to section~\ref{app:taco_details} in appendix and the online documentation.

The python class \texttt{Problem} wraps up all information about the problem to be solved.
This class possesses an attribute \texttt{data} which contains the values of $\xi$ and is formatted as a \texttt{numpy} array in 64-bit float precision. The class also implements two methods giving first-order oracles: \texttt{objective\_func} and \texttt{objective\_grad} for the objective function $f$, and \texttt{constraint\_func} and \texttt{constraint\_grad} for the constraint function $g$.

Let us take a simple quadratic problem in $\R^2$ to illustrate the instantiation of a problem. We consider
\begin{align*}
    &\!\!\min_{x\in \R^2}\;\; \|x - a\|^2   & a = [1.0, 2.0]^\top\\
    &\text{s.t.} \;\; \probability[{\x}^\top \xi \leq 0] \geq 0.9, & \text{ with $1000$ samples of $\xi\sim \mathcal{N}(0,1)$} .
\end{align*}
The instance of \texttt{Problem} is in this case:
\begin{lstlisting}
    import numpy as np
    class Problem:
        def __init__(self, dim=2, sample_size=1000):
            self.data = np.random.normal(size=(sample_size, dim), dtype=np.float64)
            self.a = np.array([1.0, 2.0], dtypte=np.float64)
        def objective_fun(self,x):
            return np.dot(x-self.a,x-self.a)
        def objective_grad(self, x):
            return x
        def contraint_func(self, x, z):
            return np.dot(x,z)
        def constraint_grad(self, x, z)
            return z
    problem = Problem()
\end{lstlisting}


 \texttt{TACO} handles the optimization process with a python class named \texttt{Optimizer}. Given an instance of Problem and hyper-parameters provided by the user, the class Optimizer runs an implementation of the bundle method of~\cite{de2019proximal} on the penalized problem~\eqref{eq:def_penalized_pb}. The toolbox gives the option to update the penalization parameters\;$\mu, \lambda$ along the running process to escape possible stationary points for the DC objective that are non-feasible for the chance constraint.
\begin{lstlisting}
    from taco import Optimizer
    problem = Problem()
    optimizer = Optimizer(problem, p=0.9, starting_point=np.zeros(2, dtype=np.float64), pen1=1.0, pen2=10.0)
    sol = optimizer.run()
\end{lstlisting}

\vspace{-1ex}

Customizable parameters are stored in a python dictionary, called \texttt{params}, and designed as an attribute of the class \texttt{Optimizer}. The main parameters to tune are: the safety level of probability\;\texttt{p}, the starting penalization parameters $\mu=\texttt{pen1}$ and $\lambda=\texttt{pen2}$, the starting point of the algorithm and the starting value for the proximal parameter of the bundle method.
Others parameters are filled with default values when instantiating an \texttt{Optimizer}; for instance:
\begin{lstlisting}
    custom_options = {
        'p': 0.9,
        'pen1': 1.0,
        'pen2': 10.0,
        'bund_mu_start': 50.0,
        'bund_max_size_bundle_set': 30,
    }
    custom_optimizer = Optimizer(problem, params=custom_options)
\end{lstlisting}

Some important parameters (such as the safety probability level, or the starting penalization parameters) may also be given directly to the constructor of the class \texttt{Optimizer}, when instantiating the object; as in the first example.

\section{Numerical illustrations}
\label{sec:experiments}

We illustrate our double penalisation approach implemented in the toolbox \texttt{TACO} on two problems:
a 2-dimensional quadratic problem with a non-convex chance constraint (in Section\;\ref{sec:plot}), and a family of problems with explicit solutions (in Section\;\ref{sec:family}). These proof-of-concept experiments are not meant to be extensive but to show that our approach is viable. These experiments are reproducible: the experimental framework is available on the toolbox's website.

\vspace*{-1ex}
\subsection{Visualization of convergence on a 2-d problem}\label{sec:plot}

\vspace*{-1ex}
We consider a two-dimensional toy quadratic problem in order to track the convergence of the iterates on 
the sublevel sets.
We take \cite[Ex.\;4.1]{vanAckooij2020CC} which considers an instance of problem~\eqref{eq:general_prob} with
\begin{equation}\label{eq:first_toy_prob}
\begin{split}
    f(x) &= \frac{1}{2} (x-a)^\top Q (x-a) \quad \text{\;with $a = \begin{pmatrix} 2. \\ 2. \end{pmatrix}$, $Q=\begin{pmatrix} 5.5 & 4.5 \\ 4.5 & 5.5 \end{pmatrix}$}\\
    g(x,z) &= z^\top W(x) z +  w^\top z \quad \text{with $W\!(x)\!=\!\begin{pmatrix} x_1^2 + 0.5\! & 0. \\ 0. & \!|x_2 - 1|^3 \!+\! 1 \end{pmatrix}$}\\[-1ex]
    \xi &\sim \mathcal{N}(\mu, \Sigma) \qquad\text{$10^4$ samplings with $\mu = \begin{pmatrix} 1. \\ 1. \end{pmatrix}$, $\Sigma = \begin{pmatrix} 20. & 0. \\ 0. & 20. \end{pmatrix}$}.
\end{split}
\end{equation}

For this example, \cite{vanAckooij2020CC} shows that the chance constraint is convex for large enough probability levels, but here we take a low probability level $p=0.008$ to have a non-convex chance-constraint. We can see this on Figure~\ref{imgs:path_pb1}, ploting the level sets of the objective function and the constraint function: the chance-constrained region for $p=0.008$ is delimited by a black dashed line; the optimal value of this problem is located at the star.

\begin{figure}[ht!]
    \centering
    \vspace*{-2ex}
    \includegraphics[trim=30 30 30 30, clip, width=0.45\linewidth]{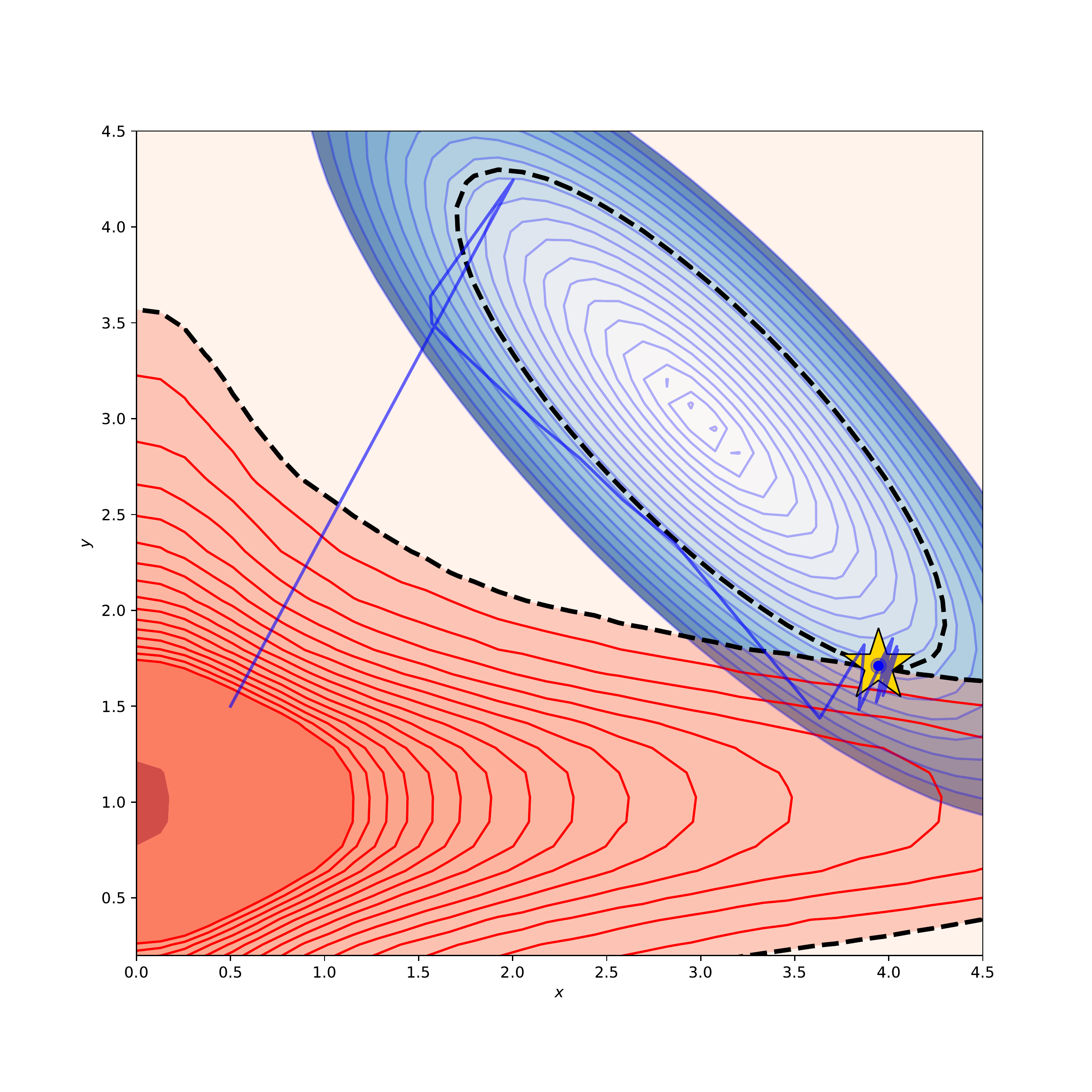}
    \vspace*{-3ex}
    \caption{Trajectory of the iterates (in blue) on the plot of the level sets of the chance-constraint and the objective for the 2d problem with data~\eqref{eq:first_toy_prob}\label{imgs:path_pb1}.}
\end{figure}

We apply our double penalization method to solve this problem, with the setting described in Appendix\;\ref{app:setting} and available on the \texttt{TACO} website.
We plot on the sublevel sets of Figure~\ref{imgs:path_pb1} the path (in deep blue) taken by the sequence of iterates starting from the point $[0.5, 1.5]$ moving towards the solution.
We observe that the sequence of iterates, after a exploration of the functions landscape, gets rapidly close to the optimal solution. At the end of the convergence, we also see a zigzag behaviour around the frontier of the chance constraint.
This can be explained by the penalization term
which is activated asymptotically whenever the sequence gets out of the chance constraint.

\vspace*{-1ex}
\subsection{Experiments on a family of problems}\label{sec:family}

We consider the family of $d$-dimensional norm problems of~\cite[section 5.1]{hong2011sequential}. For a given dimension $d$, the problem writes as an instance of\;\eqref{eq:general_prob} with
\begin{equation}\label{eq:def_pb2}
        f(x) = - \|x\|_1 \qquad\text{and}\qquad g(x, Z) = \!\!\!\! \max_{i \in \{ 1,\dots, 10 \}} \sum_{j=1}^d Z_{i,j}^2 x_j^2 - 100
\end{equation}
and $\xi$ is random matrix of dimensions $10\times d$ statisfying for all $i,j$, $\xi_{i,j} \sim \mathcal{N}(0,1)$. The interest of this family of problems is that they have explicit solutions: for given $d$, the optimal value is
\begin{equation*}
    f^\star = -\frac{10\,d}{\sqrt{F_{\chi_d^2}^{(-1)}(p^{\frac{1}{10}})}}
\end{equation*}
where $F_{\chi_d^2}$ is $\chi^2$ cumulative distribution with $d$ degrees of freedom. We consider four instances of this problems with dimension $d$ from $2$ to $200$ and the safety probability threshold $p$ set to $0.8$. We consider the case of the rich information on uncertainty: $\xi$ is sampled $10000$ times. In this case, a direct approach consisting in solving the standard
mixed-integer quadratic reformulations (see e.g.\;\cite{ahmed2008solving}) with efficient MINLP solvers (we used 
\texttt{Juniper}\;\cite{juniper}) does not provide reasonable solutions; see basic information in Appendix\;\ref{app:setting}.

\begin{figure*}[!ht]
		\centering
        {\adjincludegraphics[width=0.6\linewidth, trim=10 12 10 10, clip=true]{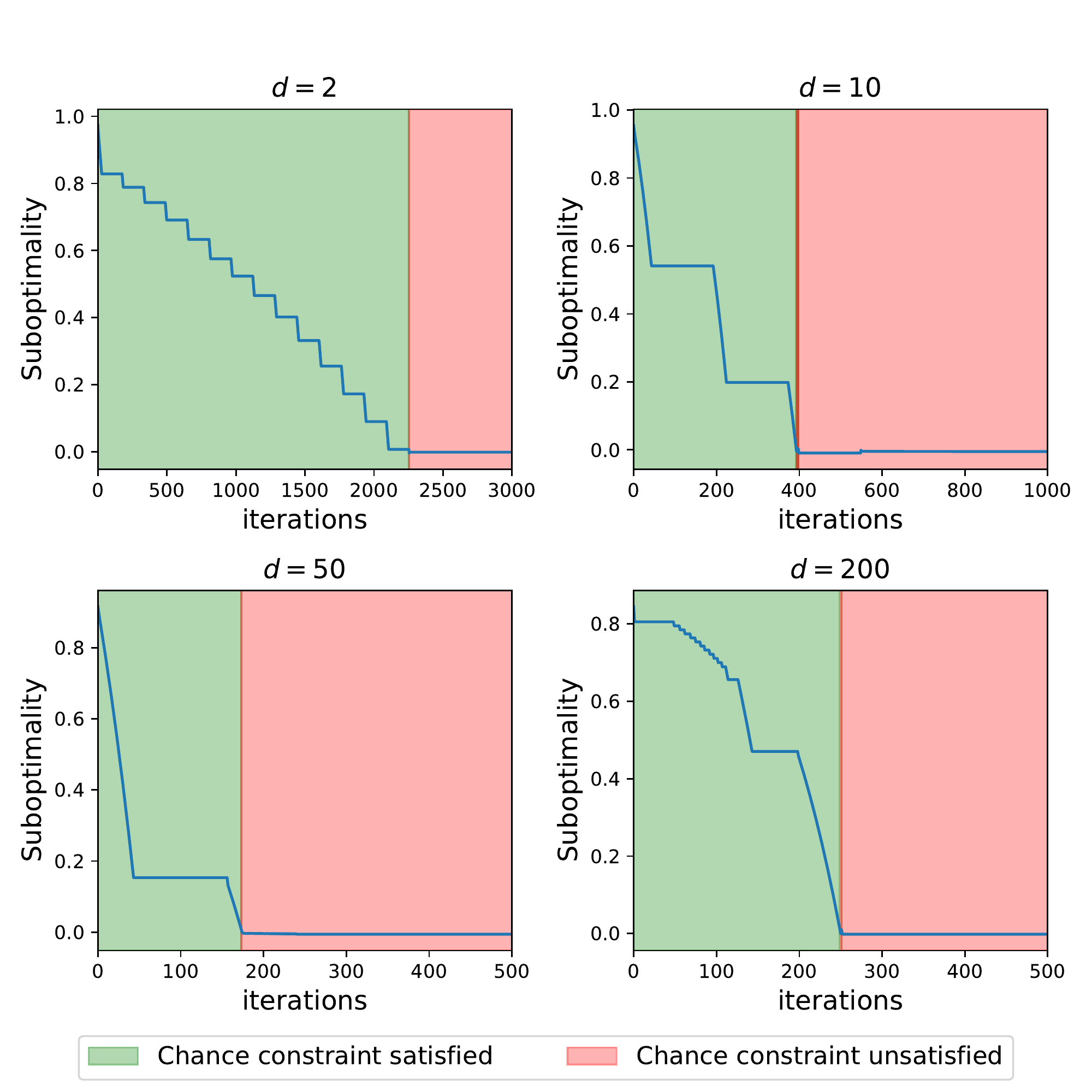}}
     \caption{\small{Convergence of the algorithm on four norm problems\;\eqref{eq:def_pb2}} with $d=2,10.50,200$.}
     \label{fig:subopt_pb2}
\end{figure*}

We solve these instances with our double penalization approach, parameterized as described in Appendix\;\ref{app:setting} (see the \texttt{TACO} website for the material and settings provided in order to reproduce experiments). Figure~\ref{fig:subopt_pb2} plots the relative suboptimality
\[
(f(x_k) - f^\star)/|f^\star|
\]
along iterations. The green (resp. red) regions represent iterates that, respectively, satisfy (resp. do not satisfy) the chance constraint.

In the four instances, we take a first iterate well inside the feasible region. We observe an initial decrease of the objective function down to optimal value. Then the chance constraint starts to be violated only when this threshold is reached, and the last part of convergence deals with local improvement of precision and feasibility.

Table~\ref{table:pb2} reports the final suboptimality and satisfaction of the probabilistic constraint. The probability constraint is evaluated for $100$ sampled points from of the total $N=10000$ points. We give the resulting probability; the standard deviation is 0.004 for the four instances.

\begin{table*}[!ht]
\begin{center}
\begin{adjustbox}{max width=0.99\linewidth}
    \begin{tabular}{lccccc}
\toprule
Dimension & Suboptimality & $\probability[g(x,\xi) \leq 0]$\\

\midrule

$d=2$ & $8.9 \times 10^{-4}$ & $0.799$\\

$d=10$ & $5.0 \times 10^{-3}$ & $0.787$\\

$d=50$ & $5.6 \times 10^{-3}$ & $0.769$\\

$d=200$ & $1.8 \times 10^{-3}$ & $0.781$\\

\bottomrule
\end{tabular}
\end{adjustbox}
\vspace*{-1ex}
\caption{Final suboptimality and feasibility for~\eqref{eq:def_pb2} (where $p=0.8$).\label{table:pb2}}
\end{center}
\end{table*}

We observe that the algorithm reaches an accuracy of order of $10^{-3}$. Regarding satisfaction of the constraint $\probability[g(x,\xi) \leq 0] \geq 0.8 $, it is achieved to a $10^{-4}$ precision for $d= 2$ but it slightly degrades as the dimension grows.

\appendix

\section{Proofs of complementary results}

\subsection{Uniform bound at the limit}\label{app:continuous_case}
We show here that the uniform error bound derived in Section~\ref{sec:analysis_vf} vanishes at the limiting case of continuous distributions. We assume that, for a fixed $x\in \Rd$, the random variable $g(x,\xi)$ has a continuous density $f_{\x, \xi}:\R \rightarrow \R$ denoted by $f_{\x, \xi}$: we have, for all $a\leq b$,
\begin{equation*}
    \probability[ a \leq g(\x,\xi)  \leq b] = \int_{a}^b f_{\x, \xi}(t) \measuredWRT t.
\end{equation*}
\vspace*{-1ex}
\begin{proposition}
    Fix $x\in \R^d$ and denote by $q_p$ the $p$-quantile of the distribution followed by the random variable $g(x,\xi)$. If  $g(x,\xi)$ has a continuous density, then the value function $\eta \mapsto h(x, \eta)$ defined in \eqref{eq:herrfct} is differentiable at $\eta = q_p$ (with $h'(x, q_p)=0$).
\end{proposition}
\begin{proof}
    We first note that the existence of a density ensures the continuity of the cumulative distribution function of $g(\x, \xi)$, which in turns implies $\probability[g(\x, \xi) \leq q_p] = p$.
    Let us now come back to expressions established in
    the proof of Theorem\;\ref{uniform_parametric_error_bound}.
    From~\eqref{eq:error_first_case}, we have, for $\eta > q_p$,
    \begin{equation*}
        \begin{split}
            h(\x, \eta) &= (\eta - q_p) \frac{1}{1-p} \left(\probability[g(\x, \xi) \leq q_p] - p + \expectation\left[\frac{\eta - g(\x, \xi)}{\eta - q_p} \one_{q_p < g(\x,\xi) \leq \eta}\right] \right) \\
                        &= \frac{1}{1-p} \expectation \left[ (\eta - g(\x, \xi)) \one_{q_p < g(\x,\xi) \leq \eta}\right]  ~=~ \frac{1}{1-p} \int_{q_p}^{\eta} (\eta - t) f_{\x, \xi}(t) \measuredWRT t \\
                        &= \frac{1}{1-p} \left(\eta \int_{q_p}^{\eta} f_{\x,  \xi}(t) \measuredWRT t - \int_{q_p}^{\eta} t f_{\x,  \xi}(t) \measuredWRT t \right). \\
        \end{split}
    \end{equation*}
    By continuity of the above integrands, we can use the fundamental theorem of calculus to get that $h(x, \cdot)$ admits a right derivative at $\eta = q_p$ such that
    \begin{equation*}
    \begin{split}
        h'_{+}(\x, \eta)&= \lim_{\substack{\eta \rightarrow q_p \\ \eta > q_p}} \frac{h(x, \eta) - h(x, q_p)}{\eta - q_p} \\
        &= \lim_{\substack{\eta \rightarrow q_p \\ \eta > q_p}} \frac{1}{1-p} \left(\eta \frac{\int_{q_p}^{\eta} f_{\x,  \xi}(t) \measuredWRT t }{\eta - q_p}  - \frac{\int_{q_p}^{\eta} t f_{\x,  \xi}(t) \measuredWRT t }{\eta - q_p} \right)\\
        &= \lim_{\substack{\eta \rightarrow q_p \\ \eta > q_p}} \frac{1}{1-p} \left(\eta f_{\x,  \xi}(q_p)  - q_p f_{\x,  \xi}(q_p) \right) ~ = ~0\,.
    \end{split}
    \end{equation*}
    For the case $\eta < q_p$, we have from~\eqref{eq:error_second_case}, together with~$\probability[g(\x, \xi) = q_p] = 0$:
    \begin{equation*}
        \begin{split}
            h(\x, \eta) &= (q_p - \eta) \frac{1}{1-p} \left( \expectation \left[ 1 - \frac{(q_p - g(\x, \xi) )}{q_p - \eta} \one_{\eta< g(\x,\xi) < q_p}\right] + \probability[g(\x, \xi) = q_p] \right)\\
            &= \frac{1}{1-p} \left( (\eta - q_p) \int_{\eta}^{q_p} f_{\x,  \xi}(t) \measuredWRT t -  \int_{\eta}^{q_p} (q_p -  t) f_{\x, \xi}(t)) \measuredWRT t   \right).
        \end{split}
    \end{equation*}
    Using again to the fundamental theorem of calculus, we get that $h(x,\cdot)$ admits a left derivative at $\eta = q_p$ with:
    \begin{equation*}
    \begin{split}
        h'_{-}(\x, \eta) &=
        \lim_{\substack{\eta \rightarrow q_p \\ \eta < q_p}} \frac{h(\x, \eta) - h(\x, q_p)}{\eta - q_p} \\
        &= \lim_{\substack{\eta \rightarrow q_p \\ \eta < q_p}}  \frac{1}{1-p} \left( (\eta - q_p) \frac{\int_{\eta}^{q_p} f_{\x,  \xi}(t) \measuredWRT t}{\eta - q_p}  - \frac{\int_{\eta}^{q_p} (q_p -  t) f_{\x, \xi}(t) \measuredWRT t}{\eta - q_p} \right) ~=~ 0\,.
    \end{split}
    \end{equation*}
    We can conclude that $h(x, \cdot)$ is differentiable at $q_p$ with zero as derivative.\qed
\end{proof}

\subsection{Proof of the subgradient explicit expressions}

We provide here a direct proof of the subgradient expressions of Proposition~\ref{prop:oracle_computation}. Let $(x, \eta) \in \setX \times \R$ be fixed, and consider first the case of $\varphi_1$. For $i \in \{1, \dots, n\}$, by successive applications of Theorems 4.1.1 and 4.4.2 from \cite[Chap.\;D]{hiriart2013convex} to the functions
\[\varphi_1^{(i)}: (x , \eta) \mapsto \frac{1}{n} \left[f(x) + \mu \max(\eta, 0) + \lambda \left(\eta + \frac{1}{1-p} \max(g(x, \xi_i) - \eta, 0) \right)\right]
\]
we get for any $i \in \{1, \dots, n\}$
\[
\begin{split}
     \frac{1}{n} s_f + \frac{\lambda}{n(1-p)} \mathds{1}_{g(x, \xi_i) > \eta}  s_g & \in \partial_x \varphi_1^{i} (x,\eta) \\
     \frac{\mu}{n} \mathds{1}_{\eta > 0} + \frac{\lambda}{n} -  \frac{\lambda}{n(1-p)} \mathds{1}_{g(x, \xi_i) > \eta}& \in \partial_\eta \varphi_1^{i} (x,\eta). \\
\end{split}
\]
Since $\varphi_1 = \sum_{i=1}^n \varphi_1^{(i)}$, we thus have
\[\begin{split}
     \left(s_f + \frac{\lambda}{n(1-p)} \sum_{i\in I_{>\eta}} {s_g}_i ~,~ \mu \mathds{1}_{\eta > 0} +  \lambda - \lambda \frac{\#(I_{>\eta})}{n(1-p)}   \right) \in \partial \varphi_1(x, \eta) \\
\end{split}
\]
For $\varphi_2$ we need first the whole subdifferential of the function $G$, which, using above mentioned properties, writes
\begin{equation*}
\begin{split}
    \partial G(\x, \eta) =  &\left\{ \left(\frac{1}{1-p} \sum_{i=1}^n \frac{{s_g}_i}{n} (\mathds{1}_{g(x, \xi_i) > \eta} + \beta_i \mathds{1}_{g(x, \xi_i) = \eta}),  \right. \right.\\
    &  \left. \left. 1 - \frac{1}{1-p} \sum_{i=1}^n \frac{1}{n} (\mathds{1}_{g(x, \xi_i) > \eta} + \beta_i \mathds{1}_{g(x, \xi_i) = \eta}) \right)\!, \; \beta_i \in [0,1], \quad \forall i \in \{1, \dots, n\} \right\}. \\
\end{split}
\end{equation*}
By taking $\beta_i = \alpha$ (for all $i\in \{1, \dots,  n\}$) with the specific $\alpha$ given in the statement, we can zero the second term in the above expression.
Now since $\varphi_2(x, \eta) = \lambda \min_{s \in \R} G(x, s)$ with $\argmin_{s \in \R} G(x, s) = Q_p(g(x, \xi))$, we apply Corollary\;4.5.3\;of~\cite[Chap.\;D]{hiriart2013convex} to obtain a subgradient of $\varphi_2$:
\[
    s_{\varphi_2} = \left(\frac{\lambda}{n(1-p)} \left(\sum_{i \in I_{>Q_{p}(g(x,\xi)}} {s_g}_i + \alpha \sum_{i \in I_{=Q_{p}(g(x,\xi)}}  {s_g}_i \right), 0 \right)
\]
which completes the proof.

\vspace*{-1.5ex}
\section{Implementation details on \texttt{TACO}}
\label{app:taco_details}



\subsection{Further customization}\label{app:patam_details}

\vspace*{-0.7ex}
\texttt{TACO} relies on a set of hyperparameters to be provided by the user and specified

in a single dictionnary passed as an argument of the class \texttt{Optimizer}. There are two families of parameters to be specified. First, the parameters concerning the oracles $\varphi_1$ and $\varphi_2$. These are the starting penalization parameters $\lambda$ and $\mu$, the multiplicative factors to increment them along the penalization process, and the smoothing parameter of $\tilde \varphi_2$.

The second family of parameters concerns the bundle method. It gathers the  proximal parameters of the bundle method, the precision targeted, the starting point of the algorithm, the maximal size of the bundle information, and
parameters related used when restarting the bundle method (see more in the following section).

Overall the most important parameters to specify are the starting penalization parameters $\mu$ and $\lambda$  with respective keys \texttt{‘pen1’} and \texttt{‘pen2’} and the starting proximal parameter of the bundle algorithm. In the toolbox, we provide the set of parameters used in our numerical experiments. In addition of the final solution, it is possible to log the iterates, function values and time values, by calling the method with the option \texttt{logs=True}. The \texttt{verbose=True} option also allows the user to observe in real time the progression of the algorithm along the iterations.

Finally we underline that \texttt{TACO} subroutines rely on just-in-time compilation supported by \texttt{Numba}, which consistently improves the running time. Further improvements can be achieved when the instance considered can be cast as a \texttt{Numba jitclass}. The parameter \texttt{'numba'} in the input dictionnary of the associated \texttt{Optimizer} object should then be set to \texttt{True}.

\vspace*{-1.7ex}
\subsection{On the bundle algorithm}\label{app:bundle}
\vspace*{-0.8ex}
Here are some information on our implementation of the bundle algorithm of\;\cite{de2019proximal} to tackle the double penalized problem $(P_{\lambda, \mu})$ written as a DC problem.
We discuss the parameters used at various steps of the procedure.
We refer to \cite{de2019proximal} for more
details.

\vspace*{-0.5ex}
\begin{itemize}
    \item \textbf{Overall run}:  The starting point, the maximum number of iterations as well as the precision tolerance for termination may be set by the user.
\vspace*{0.5ex}

    \item \textbf{Subproblems}: Each iteration of the bundle algorithm requires solving a quadratic subproblem (see \cite[Eq.\;(9)]{de2019proximal}), for which we use the solver \texttt{cvxopt}~\cite{vandenberghe2010cvxopt} by simplicity.
    \vspace*{0.5ex}

    \item \textbf{Stabilization center}: 
    Whenever the solution of a subproblem satisfies a sufficient decrease in terms a function value, it is considered as a new stability center. The condition to qualify sufficient decrease is given in~\cite[Eq.\;(12)]{de2019proximal}. It involves a constant $\kappa$ which may be tuned by the user.
    \vspace*{0.5ex}

    \item \textbf{Proximal parameters}: The initial value of the proximal parameter involved in quadratic subproblems can be set by the user. The user can also specify  upper and lower acceptance bounds for it. After each iteration, the prox-parameter is updated: it is increased by a constant factor in case of serious step, and decreased otherwise. Both factors can be tuned by the user.
    \vspace*{0.5ex}

    \item \textbf{Bundle information}: The bundle of cutting-planes is augmented after each null step with new linearization, and emptied after each serious step.  We fix a maximum size for the bundle: above this parameter, the bundle is emptied and proximal parameter is restarted to a specified restarting value. When the bundle is emptied, we have the chance of a specific improvement: if the stability center is feasible in the chance-constraint, we replace the coordinate playing the role of $\eta$ by the $p$-quantile of $g(x,\xi)$, thus decrease the objective function.
    \vspace*{0.5ex}

    \item \textbf{Termination Criteria}: We use a simple stopping criteria: we stop  when the euclidean distance between the current iterate and the current stability center falls below a certain threshold specified by the user.
\end{itemize}

\vspace*{-3.3ex}
\subsection{Experimental settings}\label{app:setting}
\vspace*{-1ex}

\vspace*{-1ex}
\paragraph{Setting of Section\;\ref{sec:plot}.}

For this 2d problem, we use the starting point $x= (0.5,1.5)$ (and $\eta = 0.01$) well-inside the chance-constraint. The initial penalization parameters $\mu$ and $\lambda$ are respectively initialized to $400$ and $600$.
The initial proximal parameter is fixed to $38.0$ with lower and upper acceptance bounds set to $10^{-3}$ and $10^{3}$. Increasing and decreasing factors for this parameter are fixed to $1.05$ and $0.95$.
The classification rule parameter is set to $10^{-4}$.
The maximal size of the information bundle is set to $20$ and the threshold of the termination criteria is set to $10^{-7}$.

\vspace*{-1ex}
\paragraph{Setting for Section\;\ref{sec:family}.}

For any fixed dimension $d$ comprised in $\{2,10, 50, 200\}$, the algorithm is run from the starting point $(0.1, \dots, 0.1) \in \R^{d+1}$. The starting penalization parameter $\mu$, constant for the 4 instances, is set to $\mu = 10.0$. We tuned the second penalization parameter $\lambda$ along problems: we observed that $\lambda = \{1.75,1.5,1.5,2.0\}$ give good performances for the considered problems.

The starting proximal parameters is fixed to $60.0$ with lower and upper acceptance bounds set to $10^{-4}$ and $10^5$ respectively. Increasing and decreasing factors for the proximal parameter are fixed to $1.01$ and $0.99$. The classification rule parameter is set to $10^{-4}$. The maximal size of the information bundle is set to $300$.

\vspace*{-1ex}
\paragraph{Limitations of MINLP approach.}
Mixed-integer reformulation approaches (see e.g.\;\cite{ahmed2008solving}) are often considered as the state-of-the-art to solve chance constrained optimization problems by sample average approximation. Applying directly such a reformulation to Problem~\eqref{eq:def_pb2} in Section\;\ref{sec:family} leads to the equivalent mixed integer quadratic program:
\vspace*{-0.9ex}
\begin{alignat*}{2}
    &\!\!\min_{x\in \mathbb{R}^d\!,\, z\in\{0,1\}^N}\;\; &&-\sum_{i=1}^d x_i \\[-1ex]
    &~~~\text{s.t.}\;\; && \sum_{k=1}^d (\xi_{i})_{j,k}^2 x_k^2 - 100 \leq M\;z_i,\;\;\; \forall i \in \llbracket 1, N \rrbracket, ~ \forall j \llbracket 1, 10 \rrbracket\\[-1ex]
    & && \sum_{i=1}^N z_i \leq p N, ~~x \geq 0.
    \vspace*{-0.5ex}
\end{alignat*}
where $M$ is a large ``big-M" constant.
In our setting, such formulation involves $10 \times N = 100000$ quadratic constraint involving binary variables. We were not able to solve the resulting mixed-integer problem in reasonable times using the MINLP solver\;\texttt{Juniper}\;\cite{juniper} (that is based on Ipopt and JuMP). This shows that a direct application of reformulation techniques combined with reliable software failed on this problem in contrast with our approach.

\clearpage

\bibliography{optim}
\bibliographystyle{abbrvnat}
\end{document}

%% file: include.tex
\usepackage{graphicx}
\usepackage{amssymb} 
\usepackage{amsfonts}       
\usepackage{amsmath}      
\usepackage{amsthm}     
\usepackage{mathtools}
\usepackage{stmaryrd}
\usepackage{dsfont}

\newtheorem{theorem}{Theorem}
\newcommand{\x}{x}

\DeclareMathOperator{\one}{\mathds{1}} 

\DeclareMathOperator*{\N}{\mathbb{N}}
\DeclareMathOperator*{\R}{\mathbb{R}}
\DeclareMathOperator*{\Rd}{\mathbb{R}^d}
\DeclareMathOperator*{\Rm}{\mathbb{R}^m}
\DeclareMathOperator*{\setX}{\mathcal{X}}

\DeclareMathOperator{\integrable}{\mathcal{L}^1}
\DeclareMathOperator{\expectation}{\mathbb{E}}      
\DeclareMathOperator{\probability}{\mathbb{P}}      

\DeclarePairedDelimiter{\abs}{\lvert}{\rvert}       

\DeclareMathOperator{\rv}{X}        
\newcommand{\cdf}[1]{F_{#1}}        
\newcommand{\pquantile}[1]{Q_p(#1)}     
\newcommand{\psuperquantile}[1]{\bar{Q}_p(#1)}      

\DeclareMathOperator{\measuredWRT}{d} 

\DeclarePairedDelimiterXPP{\exof}[1]{\ex}{[}{]}{}{
 #1}

\DeclarePairedDelimiterXPP{\probof}[1]{\prob}{(}{)}{}{
 #1}

\DeclareMathOperator*{\argmin}{arg\,min}        

\newcommand{\indexset}{\mathcal{I}_n}
\newcommand{\Plmu}{(P_{\lambda, \mu})}
\newcommand{\dist}[1]{d_{#1}}
\usepackage{adjustbox}
\usepackage{booktabs}
\usepackage[justification=centering]{caption}
\usepackage{xcolor,colortbl}
\definecolor{Gray}{gray}{0.85}
\newcolumntype{a}{>{\columncolor{Gray}}l}
\usepackage{soul,xcolor}
\setstcolor{blue}

\usepackage{hyperref}

\usepackage{listings}
\usepackage{color}

\definecolor{dkgreen}{rgb}{0,0.6,0}
\definecolor{gray}{rgb}{0.5,0.5,0.5}
\definecolor{mauve}{rgb}{0.58,0,0.82}

\makeatletter
\newcommand{\fixed@sra}{$\vrule height 2\fontdimen22\textfont2 width 0pt\shortrightarrow$}
\newcommand{\shortarrow}[1]{%
  \mathrel{\text{\rotatebox[origin=c]{\numexpr#1*45}{\fixed@sra}}}
}

\newtheorem{lemma}[theorem]{Lemma}
\newtheorem{proposition}[theorem]{Proposition}

\usepackage{cleveref}
\usepackage{hyperref}

%% file: main.bbl
\begin{thebibliography}{33}
\providecommand{\natexlab}[1]{#1}
\providecommand{\url}[1]{\texttt{#1}}
\expandafter\ifx\csname urlstyle\endcsname\relax
  \providecommand{\doi}[1]{doi: #1}\else
  \providecommand{\doi}{doi: \begingroup \urlstyle{rm}\Url}\fi

\bibitem[Ahmed and Shapiro(2008)]{ahmed2008solving}
S.~Ahmed and A.~Shapiro.
\newblock Solving chance-constrained stochastic programs via sampling and
  integer programming.
\newblock In \emph{State-of-the-art decision-making tools in the
  information-intensive age}, pages 261--269. Informs, 2008.

\bibitem[Ben-Tal and Teboulle(2007)]{ben2007old}
A.~Ben-Tal and M.~Teboulle.
\newblock An old-new concept of convex risk measures: The optimized certainty
  equivalent.
\newblock \emph{Mathematical Finance}, 17\penalty0 (3):\penalty0 449--476,
  2007.

\bibitem[Calafiore and Campi(2006)]{calafiore2006scenario}
G.~C. Calafiore and M.~C. Campi.
\newblock The scenario approach to robust control design.
\newblock \emph{IEEE Transactions on Automatic Control}, 51\penalty0
  (5):\penalty0 742--753, 2006.

\bibitem[Charnes and Cooper(1959)]{charnes1959chance}
A.~Charnes and W.~W. Cooper.
\newblock Chance-constrained programming.
\newblock \emph{Management science}, 6\penalty0 (1):\penalty0 73--79, 1959.

\bibitem[Chow et~al.(2017)Chow, Ghavamzadeh, Janson, and Pavone]{chow2017risk}
Y.~Chow, M.~Ghavamzadeh, L.~Janson, and M.~Pavone.
\newblock Risk-constrained reinforcement learning with percentile risk
  criteria.
\newblock \emph{The Journal of Machine Learning Research}, 18\penalty0
  (1):\penalty0 6070--6120, 2017.

\bibitem[Clarke(1990)]{clarke1990optimization}
F.~H. Clarke.
\newblock \emph{Optimization and nonsmooth analysis}, volume~5.
\newblock Siam, 1990.

\bibitem[de~Oliveira(2019)]{de2019proximal}
W.~de~Oliveira.
\newblock Proximal bundle methods for nonsmooth dc programming.
\newblock \emph{Journal of Global Optimization}, 2019.

\bibitem[Dentcheva(2009)]{Dentcheva_2009}
D.~Dentcheva.
\newblock Optimisation models with probabilistic constraints.
\newblock In A.~Shapiro, D.~Dentcheva, and A.~Ruszczy\'{n}ski, editors,
  \emph{Lectures on Stochastic Programming. Modeling and Theory}, volume~9 of
  \emph{MPS-{SIAM} series on optimization}. SIAM, 2009.

\bibitem[Dentcheva et~al.(2000)Dentcheva, Pr{\'e}kopa, and
  Ruszczynski]{dentcheva2000concavity}
D.~Dentcheva, A.~Pr{\'e}kopa, and A.~Ruszczynski.
\newblock Concavity and efficient points of discrete distributions in
  probabilistic programming.
\newblock \emph{Mathematical Programming}, 89\penalty0 (1), 2000.

\bibitem[F{\"o}llmer and Schied(2002)]{follmer2002convex}
H.~F{\"o}llmer and A.~Schied.
\newblock Convex measures of risk and trading constraints.
\newblock \emph{Finance and stochastics}, 6\penalty0 (4):\penalty0 429--447,
  2002.

\bibitem[Henrion and Strugarek(2008)]{Henrion_Strugarek_2008}
R.~Henrion and C.~Strugarek.
\newblock Convexity of chance constraints with independent random variables.
\newblock \emph{Computational Optimization and Applications}, 41:\penalty0
  263--276, 2008.

\bibitem[Hiriart-Urruty and Lemar{\'e}chal(2013)]{hiriart2013convex}
J.-B. Hiriart-Urruty and C.~Lemar{\'e}chal.
\newblock \emph{Convex analysis and minimization algorithms I: Fundamentals},
  volume 305.
\newblock Springer science \& business media, 2013.

\bibitem[Hong et~al.(2011)Hong, Yang, and Zhang]{hong2011sequential}
L.~J. Hong, Y.~Yang, and L.~Zhang.
\newblock Sequential convex approximations to joint chance constrained
  programs: A monte carlo approach.
\newblock \emph{Operations Research}, 59\penalty0 (3), 2011.

\bibitem[Kröger et~al.(2018)Kröger, Coffrin, Hijazi, and Nagarajan]{juniper}
O.~Kröger, C.~Coffrin, H.~Hijazi, and H.~Nagarajan.
\newblock Juniper: An open-source nonlinear branch-and-bound solver in julia.
\newblock In \emph{Integration of Constraint Programming, Artificial
  Intelligence, and Operations Research}. Springer International Publishing,
  2018.
\newblock ISBN 978-3-319-93031-2.

\bibitem[Laguel et~al.(2020)Laguel, Malick, and
  Harchaoui]{laguel-etal:spqr:mlsp2020}
Y.~Laguel, J.~Malick, and Z.~Harchaoui.
\newblock First-order optimization for superquantile-based supervised learning.
\newblock In \emph{2020 IEEE 30th International Workshop on Machine Learning
  for Signal Processing (MLSP)}, pages 1--6. IEEE, 2020.

\bibitem[Lam et~al.(2015)Lam, Pitrou, and Seibert]{10.1145/2833157.2833162}
S.~K. Lam, A.~Pitrou, and S.~Seibert.
\newblock Numba: A llvm-based python jit compiler.
\newblock In \emph{Proceedings of the Second Workshop on the LLVM Compiler
  Infrastructure in HPC}, LLVM '15, New York, NY, USA, 2015. Association for
  Computing Machinery.
\newblock ISBN 9781450340052.

\bibitem[Luedtke and Ahmed(2008)]{Luedtke_Ahmed_2008}
J.~Luedtke and S.~Ahmed.
\newblock A sample approximation approach for optimization with probabilistic
  constraints.
\newblock \emph{{SIAM} Journal on Optimization}, 19:\penalty0 674--699, 2008.

\bibitem[Luenberger and Ye(1984)]{luenberger1984linear}
D.~G. Luenberger and Y.~Ye.
\newblock \emph{Linear and nonlinear programming}, volume~2.
\newblock Springer, 1984.

\bibitem[Medova(1998)]{medova1998chance}
E.~Medova.
\newblock Chance-constrained stochastic programming forintegrated services
  network management.
\newblock \emph{Annals of Operations Research}, 81:\penalty0 213--230, 1998.

\bibitem[Nemirovski and Shapiro(2006)]{nemirovski2006convex}
A.~Nemirovski and A.~Shapiro.
\newblock Convex approximations of chance constrained programs.
\newblock \emph{SIAM Journal on Optimization}, 17\penalty0 (4):\penalty0
  969--996, 2006.

\bibitem[Nesterov(2005)]{nesterov2005smooth}
Y.~Nesterov.
\newblock Smooth minimization of non-smooth functions.
\newblock \emph{Mathematical programming}, 103\penalty0 (1):\penalty0 127--152,
  2005.

\bibitem[Pagnoncelli et~al.(2009)Pagnoncelli, Ahmed, and
  Shapiro]{pagnoncelli2009sample}
B.~K. Pagnoncelli, S.~Ahmed, and A.~Shapiro.
\newblock Sample average approximation method for chance constrained
  programming: theory and applications.
\newblock \emph{Journal of optimization theory and applications}, 142\penalty0
  (2):\penalty0 399--416, 2009.

\bibitem[Pr{\'{e}}kopa(1995)]{Prekopa_1995}
A.~Pr{\'{e}}kopa.
\newblock \emph{Stochastic Programming}.
\newblock Kluwer, Dordrecht, 1995.
\newblock \doi{10.1007/978-94-017-3087-7}.

\bibitem[Pr{\'e}kopa and Sz{\'a}ntai(1978)]{prekopa1978flood}
A.~Pr{\'e}kopa and T.~Sz{\'a}ntai.
\newblock Flood control reservoir system design using stochastic programming.
\newblock In \emph{Mathematical programming in use}, pages 138--151. Springer,
  1978.

\bibitem[Rockafellar and Royset(2013)]{rockafellar2013superquantiles}
R.~T. Rockafellar and J.~O. Royset.
\newblock Superquantiles and their applications to risk, random variables, and
  regression.
\newblock In \emph{Theory Driven by Influential Applications}. INFORMS, 2013.

\bibitem[Rockafellar and Royset(2014)]{rockafellar2014random}
R.~T. Rockafellar and J.~O. Royset.
\newblock Random variables, monotone relations, and convex analysis.
\newblock \emph{Mathematical Programming}, 148\penalty0 (1-2):\penalty0
  297--331, 2014.

\bibitem[Rockafellar and Uryasev(2000)]{rockafellar2000optimization}
R.~T. Rockafellar and S.~Uryasev.
\newblock Optimization of conditional value-at-risk.
\newblock \emph{Journal of risk}, 2:\penalty0 21--42, 2000.

\bibitem[Ruszczy{\'n}ski and Shapiro(2006)]{ruszczynski2006optimization}
A.~Ruszczy{\'n}ski and A.~Shapiro.
\newblock Optimization of convex risk functions.
\newblock \emph{Mathematics of operations research}, 31\penalty0 (3):\penalty0
  433--452, 2006.

\bibitem[{van Ackooij}(2020)]{vanAckooij_2020}
W.~{van Ackooij}.
\newblock A discussion of probability functions and constraints from a
  variational perspective.
\newblock \emph{Set-Valued and Variational Analysis (online)}, 28:\penalty0
  585–609, 2020.

\bibitem[{van Ackooij} et~al.(2014){van Ackooij}, Henrion, M\"{o}ller, and
  Zorgati]{vanAckooij_Henrion_Moller_Zorgati_2011b}
W.~{van Ackooij}, R.~Henrion, A.~M\"{o}ller, and R.~Zorgati.
\newblock Joint chance constrained programming for hydro reservoir management.
\newblock \emph{Optimization and Engineering}, 15, 2014.

\bibitem[Van~Ackooij et~al.(2020)Van~Ackooij, Laguel, Malick, and
  Matiussi-Ramalho]{vanAckooij2020CC}
W.~Van~Ackooij, Y.~Laguel, J.~Malick, and G.~Matiussi-Ramalho.
\newblock On the convexity of level-sets of probability functions.
\newblock \emph{Submitted}, 2020.

\bibitem[Vandenberghe(2010)]{vandenberghe2010cvxopt}
L.~Vandenberghe.
\newblock The cvxopt linear and quadratic cone program solvers.
\newblock \emph{Online: http://cvxopt. org/documentation/coneprog. pdf}, 2010.

\bibitem[Ye et~al.(1997)Ye, Zhu, and Zhu]{ye1997exact}
J.~J. Ye, D.~Zhu, and Q.~J. Zhu.
\newblock Exact penalization and necessary optimality conditions for
  generalized bilevel programming problems.
\newblock \emph{SIAM Journal on optimization}, 7\penalty0 (2):\penalty0
  481--507, 1997.

\end{thebibliography}
